\documentclass[onefignum,onetabnum]{siamart171218}
\usepackage{amstext}
\usepackage{amsfonts}
\usepackage[T1]{fontenc}
\usepackage[utf8]{inputenc}
\usepackage{mathrsfs}
\usepackage{graphicx,epstopdf}
\numberwithin{equation}{section}
\numberwithin{figure}{section}

\newtheorem{claim}[theorem]{Claim}
\newtheorem{conjecture}[theorem]{Conjecture}

\newcommand{\ES}{Erd\H{o}s-S\'{o}s }
\newcommand{\LKS}{Loebl-Koml\'{o}s-S\'{o}s }

\newcommand{\fD}{\mathcal{D}}

\newcommand{\dom}{\mathrm{dom}}
\newcommand{\reg}{{T\ref{thm:regurality_lemma}}}
\newcommand{\adeg}{\mathrm{d\overline{eg}}}
\newcommand{\dens}{d}

\newcommand{\eps}{\varepsilon}

\newcommand{\dist}{\mathrm{dist}}
\renewcommand{\phi}{\varphi}
\renewcommand{\epsilon}{\varepsilon}
\renewcommand{\bf}{\mathbf}

\newcommand{\JUSTIFY}[1]{\mbox{\fbox{\tiny#1}}\quad}

\newcommand\abs[1]{\left|#1\right|}

\title{A local approach to the Erd\H{o}s-S\'{o}s conjecture\thanks{The author was supported by the Czech Science Foundation, grant number GJ16-07822Y.}}

\author{VÁCLAV ROZHOŇ\thanks Faculty of Mathematics and Physics, Charles University \& The Czech Academy of Sciences, Institute of Computer Science, Pod Vod\'{a}renskou v\v{e}\v{z}\'{\i} 2, 182 07 Prague, Czech Republic. With institutional support RVO:67985807. 
  (\email{vaclavrozhon@gmail.com}).}

\begin{document}
\maketitle
\begin{abstract}
A famous conjecture of Erd{\H o}s and Sós states that every graph with average degree more than $k - 1$ contains all trees with $k$ edges as subgraphs. 
We prove that the Erd{\H o}s-Sós conjecture holds approximately, if the size of the embedded tree is linear in the size of the graph, and the maximum degree of the tree is sublinear.

\end{abstract}


\section{Introduction}
\label{sec:intro}
Typical problems in extremal graph theory ask, how many edges in a graph force it to contain a given subgraph. A classical example of a result in this area is Turán’s Theorem, which determines the average degree that guarantees the containment of the complete graph $K_r$. 
A more complex example is the Erd{\H o}s-Stone Theorem~\cite{Erdos1946}, which essentially determines the average degree condition guaranteeing that the host graph contains a fixed non-bipartite graph. 
On the other hand, for a general bipartite graph the problem is wide open. 
If the embedded graph is a tree, the celebrated conjecture of Erd{\H o}s and Sós asserts that an average degree greater than $k - 1$ forces a copy of any tree of order $k + 1$. 

\begin{conjecture}[The Erd\H{o}s-Sós conjecture]
\label{conj:es}
Every graph $G$ with $\adeg(G) > k-1$ contains any tree on $k+1$ vertices. 
\end{conjecture}
Here $\adeg(G)$ means the average degree of $G$; similarly, we denote the minimum and the maximum degree of $G$ by $\delta(G)$ and $\Delta(G)$, respectively.

Observe that the conjecture is optimal, since a graph with average degree at most $k-1$ may have only $k$ vertices. Also observe that if we replace the condition on the average degree by a stronger condition $\delta(G) > k-1$, the conjecture becomes trivial, since we can embed any tree on $k+1$ vertices in $G$ in a greedy manner -- every time we embed a vertex of the tree such that its neighbour is already embeded; since the neighbourhood of the already embedded vertex is sufficiently large, we may always do that. Note that each graph with average degree $\adeg(G) \ge 2k$ contains a subgraph with $\delta(G) \ge k$. Such a subgraph can be found by repeatedly deleting vertices of degree smaller than $k$ from $G$. Hence, the conjecture also holds trivially if we allow ourselves to lose a factor of $2$. 

After one verifies that the \ES conjecture is true for both trees of diameter at most three, and for paths (this was done already by Erd{\H o}s and Gallai in 1959 \cite{Erdos1959}) one can observe that such trees can be embedded even in the case when the host graph contains a vertex of degree at least $k$ and its minimum degree is at least $k/2$. This is trivial for trees of diameter at most three, while for the case of paths this follows from the mentioned proof of Erd{\H o}s and Gallai. 

While this local condition on the minimum and maximum degree of $G$ suffices for both of these special cases, it already fails for trees of diameter four, as is demonstrated by the following example from \cite{Havet2018+}. Let $T$ be a tree consisting of a vertex connected to centres of three stars on $k/3$ vertices and let $G$ be a graph consisting of a vertex complete to either two cliques of size $k/2$, or $K_{k/2, k/2}$. 
Then $\Delta(G) \ge k$ and $\delta(G) \ge k/2$, but $T$ is not contained in $G$ (see Figure \ref{fig:broom}). 
This example shows that it would be na\"ive to try to prove the \ES conjecture in the most general setting using only the local consequence of the bound on the average degree on the maximum and minimum degree of $G$. 
We will actually show in Section \ref{subsec:restrict} that trees of diameter at most three and paths are special cases. Proposition~\ref{prop:random_trees} states that with high probability, a random tree on $k+1$ vertices cannot be embedded in the host graph with two cliques from Figure \ref{fig:broom}. 

\begin{figure}
    \centering
    \includegraphics[width=\textwidth]{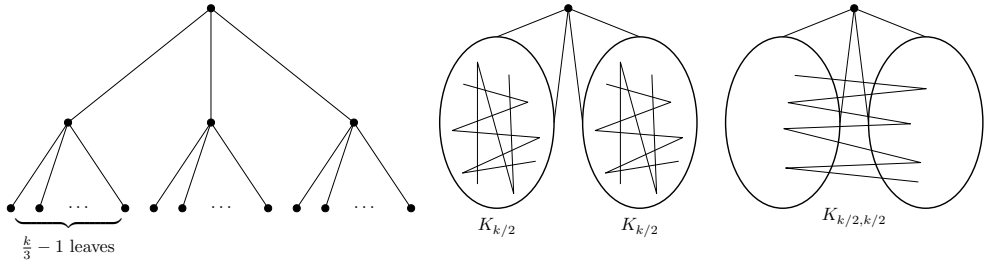}
    \caption{A tree on $k+1$ vertices and two host graphs of the same size showing that there are graphs with $\Delta(G)=k$ and $\delta(G)\ge k/2$ that do not contain a tree on $k+1$ edges. 
    The tree consists of a vertex connected to centres of three stars with $\frac{k}{3}-1$ leaves. The two graphs consist of a vertex complete to either to complete graphs on $\frac{k}{2}$ vertices or a complete bipartite graph with two colour classes with $\frac{k}{2}$ vertices. 
    The example is taken from \cite{Havet2018+}. }
    \label{fig:broom}
\end{figure}

Despite this fact, we devote this paper to this \emph{local} approach to the \ES conjecture. We are interested in the following question: if the host graph contains a certain amount of vertices of degree at least roughly $k$ and its minimum degree is at least roughly $k/2$, can we embed all trees of order $k$ in the graph?
We partially answer this question for dense graphs, which in turn implies an approximate version of the \ES conjecture. 

\subsection{Main result of this paper}

The main positive result of this paper is the following theorem. 

\begin{theorem}
\label{thm:ES_dense_bounded}
For any $\eta > 0$ there exists $n_0$ and $\gamma > 0$ such that for every $n > n_0$ and $k > 0$, any graph of order $n$ with average degree $\adeg(G)\ge k + \eta n$ contains every tree on $k$ vertices with maximum degree $\Delta(T) \le \gamma k$. 
\end{theorem}

Another way to state this result is the following. 
Suppose that we have a class of trees $\mathcal T$ such that there exits a function $f, f(n) \in o(n)$ with $\forall T \in \mathcal T: \Delta(T) \le f(|T|)$. Then there is another function $g, g(n) \in o(n)$ such that every graph with average degree $k + g(n)$ contains any tree from $\mathcal T$ on at most $k$ vertices. 

Theorem~\ref{thm:ES_dense_bounded} is trivial if $k \le \eta n/2$, since, as we already mentioned, a graph with average degree at least $\eta n$ contains a subgraph with minimum degree at least $\eta n/2$ and we can then embed $T$ greedily in the subgraph. 
Hence, we interpret this result as one for trees of size linear in the size of the host graph: only for trees of size linear in the size of the host graph this result is nontrivial. 
Another viewpoint is to consider this as a result for dense graphs: the additive error term $+\eta n$ ensures that $G$ contains at least $\eta n^2/2$ edges, hence it is dense.


We in fact prove a stronger result -- Theorem~\ref{thm:localES_dense_skew} -- that we first describe less formally. 
If we fix $r=1/2$ in the statement of Theorem~\ref{thm:localES_dense_skew}, then it states that one can embed any tree with $k$ vertices from a class of trees with sublinear maximum degree in every large enough host graph that fulfils two conditions:
\begin{enumerate}
    \item there is positive proportion of vertices of degree at least roughly $k$,
    \item the minimum degree is at least roughly $k/2$. 
\end{enumerate}
Additionally we get the following trade-off. If our class of trees satisfies that the size of the smaller colour class of any tree $T$ from the class is bounded by $r|T|$, it suffices to assume that the minimum degree of the host graph is at least roughly $rk$. Also, only the vertices of the bigger colour class are required to have bounded maximum degree. 


\begin{theorem}
\label{thm:localES_dense_skew}
For any $r, \eta > 0$ there exist $n_0$ and $\gamma > 0$ such that the following holds. 
Let $G$ be a graph of order $n > n_0$ and $T$ a tree of order $k$ with two colour classes $V_1, V_2$ such that $|V_1| \le rk$ and $\forall v \in V_2 : \deg(v) \le \gamma k$. 
If $\delta(G) \ge rk + \eta n$, and at least $\eta n$ vertices of $G$ have degree at least $k + \eta n$, then $G$ contains $T$. 
\end{theorem}

As in the case of Theorem~\ref{thm:ES_dense_bounded}, this theorem is nontrivial only for trees of size linear in the size of the host graph. 

We  postpone a simple reduction of Theorem~\ref{thm:ES_dense_bounded} to Theorem~\ref{thm:localES_dense_skew}, as well as some further remarks regarding the theorem, to Section \ref{sec:discuss}.   

We believe that Theorem \ref{thm:ES_dense_bounded} for $r=1/2$ can be substantially generalised and put this generalisation as a conjecture. 
Motivation of this conjecture is the question of how many vertices of degree at least $k$ are needed to embed any tree on $k+1$ vertices, if we moreover assume that the minimum degree is at least $k/2$. 
\begin{conjecture} (Klimošová, Piguet, Rozhoň)
    \label{conj:dense_big_degree}
    Every graph $G$ on $n$ vertices with $\delta(G) \ge k/2$ and at least $\frac12 \cdot \frac{n}{\sqrt{k}}$ vertices of degree at least $k$ contains every tree of order $k+1$. 
\end{conjecture}
We postpone to Section \ref{sec:discuss} the construction that shows that the constant $\frac12$ in Conjecture \ref{conj:dense_big_degree} cannot be improved. 


\subsection{Relation of this paper to other work}

There are many partial results concerning the \ES conjecture. It has been verified for some special families of host graphs \cite{ Balasub2007, Brandt1996, Dobson2002, Sacle1997, Wang2000}, special families of trees embedded \cite{Fan2013, Fan2007, McLennan2005}, or when the size of the host graph is only slightly larger than the size of the tree \cite{Gorlich2016, Tiner2010, Wozniak1996}. 

A solution of this conjecture for large $k$, based on an extension of the regularity lemma, has been announced in the
early 1990’s by Ajtai, Komlós, Simonovits, and Szemerédi. This result will be published as a sequence of three papers \cite{Ajtai1,Ajtai3,Ajtai2}. Although Theorem~\ref{thm:ES_dense_bounded} is only a special case of this announced result, we still believe that it is of interest, since its proof is relatively straightforward. 

A similar approach to ours was recently independently used in \cite{Besomi18a} to obtain a result very similar to Theorem~\ref{thm:ES_dense_bounded}. The only difference is that the authors of \cite{Besomi18a} require the maximum degree of the respective tree to be less than $k^{\frac{1}{67}}$ instead of $o(k)$ as in Theorem~\ref{thm:ES_dense_bounded}. The follow-up work \cite{Besomi18+a, Besomi18+b} then independently proves Theorem~\ref{thm:ES_dense_bounded}.

The idea to study embedding of trees under conditions on the minimum and maximum degree comes from the paper \cite{Havet2018+} and was later developed in \cite{Besomi18+a, Besomi18+b, Besomi18a, Besomi18+c}. 
In a preliminary version of this paper we conjectured in this direction, together with Klimošová and Piguet, that any graph $G$ that satisfies $\Delta(G) \ge 4k/3$ and $\delta(G) \ge k/2$ embeds any tree on $k+1$ vertices, but this was shown to be false in \cite{Besomi18+c}.

To prove Theorem \ref{thm:localES_dense_skew} we employ the so-called regularity method, which is a relatively standard method used for embedding trees. 
This method was successfully used for proofs of similar results, notably for the sequence of results \cite{Cooley2009, HladkyLKS1, HladkyLKS2, HladkyLKS3, HladkyLKS4, Hladkyn, Piguet2012, Zhao2011} on the so-called \LKS conjecture that asserts that each graph containing at least half of vertices of degree at least $k$ contains any tree on $k+1$ vertices. 

Simonovits conjectured that one can generalise the statement of the \LKS conjecture. The Simonovits' conjecture 
states that if one assumes that only $r|G|$ vertices ($0 \le r \le\frac12$) of the host graph $G$ have degree at least $k$, one can still embed in $G$ all trees of order $k+1$ such that their smaller colour class has size at most $r(k+1)$. The dense approximate version (i.e., the necessary degree is replaced by $k + \eta n$ instead of just $k$) of this conjecture was proven in \cite{Klimosova2018+}. 
The authors used the word \textit{skew} to denote the ratio $r$ of the size of the smaller colour class of the tree and the size of the tree itself. We adopt this notation. 

Note that the spirit of the Simonovits' conjecture is the same as the spirit of Theorem~\ref{thm:localES_dense_skew}, i.e., one can generalise an embedding theorem~by considering the skew of the embedded tree as an additional parameter. Highly skewed trees (trees with very small $r$) can then (at least in the dense approximate setting) be embedded under much milder conditions than general trees. 
This suggests that the skew of the embedded tree could be considered as a natural parameter showing how hard it is to embed the given tree. 
To give an example: the star is an extremely skewed tree that is very easy to embed. On the other hand, the example tree from Figure \ref{fig:broom} is also highly skewed, albeit it is a hard example for certain embedding setting. 
The usefulness of this parameter is thus yet to be determined. 

\subsection{Organisation of the paper}

The paper is organised as follows. In the next section we prove Theorem \ref{thm:ES_dense_bounded} and provide several remarks and constructions relevant to the results mentioned in the introduction.  In Section \ref{sec:techniques} we explain standard tools that we later use for the proof of Theorem~\ref{thm:localES_dense_skew}. Finally, in Section \ref{sec:proof} we prove Theorem~\ref{thm:localES_dense_skew}.

\section{Proof of Theorem \ref{thm:ES_dense_bounded} and further remarks}
\label{sec:discuss}

In this section we prove Theorem \ref{thm:ES_dense_bounded} and then further elaborate on several topics already mentioned in the introduction. 

\label{subsec:redukce}


\begin{proof}[Proof of Theorem \ref{thm:ES_dense_bounded}]
Let $\eta' = \eta/2$ and let $G$ be a graph on $n \ge n_0=\frac{n_{0, T\ref{thm:localES_dense_skew}}(\eta')}{\eta}$ vertices. Here $n_{0, T\ref{thm:localES_dense_skew}}(\eta')$ means the output of Theorem \ref{thm:localES_dense_skew} with input $\eta'$ and $r=1/2$. Suppose that $k \ge \eta n /2$. 

We choose a subgraph $G' \subseteq G$ such that $\adeg(G') \ge k + \eta n$ and $\delta(G') \ge k/2 + \eta n/2$. Hence, we know that the size of $G'$ is at least $k + \eta n \ge \eta n \ge n_{0, T\ref{thm:localES_dense_skew}}$. 

We claim that at least $\eta' |G'|$ vertices of $G'$ have degree at least $k + \eta' n$ and hence we may apply Theorem \ref{thm:localES_dense_skew}. If this was not true, most of the vertices of $G'$ would have degree less than $k + \eta' n$ and hence
    \[
        \adeg(G')
        \leq \eta' \cdot n + (1-\eta') \cdot (k+\eta'n)
        < \eta' n + (k+\eta' n)
        = k + 2\eta' n
        = k + \eta n,
    \]
a contradiction.
\end{proof}

\subsection{A graph from Figure \ref{fig:broom} fails to embed a random tree}
\label{subsec:restrict}
We observe that the example graph with two cliques from Figure \ref{fig:broom} fails to embed not only the tree from the same figure, but it actually fails to embed most trees.  
\begin{proposition} [Stephan Wagner, personal communication]
\label{prop:random_trees}
For even $k$ the probability that a random unlabelled tree of size $k+1$ can be embedded in the graph $G$ consisting of a vertex complete to two cliques of size $k/2$ is in $ O(k^{-1/2})$. 
\end{proposition}
\begin{proof}
We at first classify trees on $k+1$ vertices that can be embedded in $G$. A vertex $u \in T$ is a centroid, if after removing it from $T$ we obtain a family of trees such that each tree is of size at most $k/2$.  Since the size of the graph is the same as the size of the tree that we embed, only a centroid of $T$ can be embedded in the vertex of $G$ complete to all other vertices. Since $k+1$ is odd, the centroid of the tree is unique. Hence, $T$ can be embedded if and only if the subtrees created after removing its centroid can be partitioned into two classes such that the number of vertices in each class is $k/2$. We call such trees \emph{balanced}. 


Let $r_k$ be the number of unlabelled rooted trees with $k$ vertices. A formula of Otter (see e.g. page 481 of \cite{Flajolet2009}) states that $r_k = \Theta( k^{-3/2} \cdot B^k )$ for some positive constant $B$. Similarly, the number of unlabelled unrooted trees $s_k$ is in $\Theta ( k^{-5/2} \cdot B^k )$ for the same constant $B$ (again page 481 of \cite{Flajolet2009}). 

Note that the number of balanced trees of order $k+1$ is at most $r_{k/2+1}^2$, since each such tree can be decomposed into two rooted trees with $k/2+1$ vertices each. Hence the number of balanced trees is in $O( k^{-3} B^k )$. Comparing this with the sequence $s_k$, we conclude that the probability that a random unlabelled tree is balanced goes to $0$ at a rate of at least $k^{-1/2}$. 
\end{proof}

\subsection{Remark about tightness of Theorem \ref{thm:localES_dense_skew}}

Although the condition on the maximum degree $\Delta(T)$ in Theorem \ref{thm:ES_dense_bounded} is probably not necessary, it is crucial for Theorem \ref{thm:localES_dense_skew}. We show in the following Claim that Theorem \ref{thm:localES_dense_skew} fails when we drop the assumption on the sublinear degree of $T$. 

\begin{claim}
Suppose that $0 < r < \frac13$. Then there is $\eta > 0$ such that there is a graph $G$ on $n$ vertices and a tree $T$ on $k$ vertices with the following properties. 
The minimum degree of $G$ is at least $rk + \eta n$ and it contains at least $\eta n$ vertices of degree $k + \eta n$. One colour class of $T$ contains at most $rk$ vertices. Finally, $G$ does not embed $T$. 
\end{claim}

\begin{figure}
    \centering
    \includegraphics[width=\textwidth]{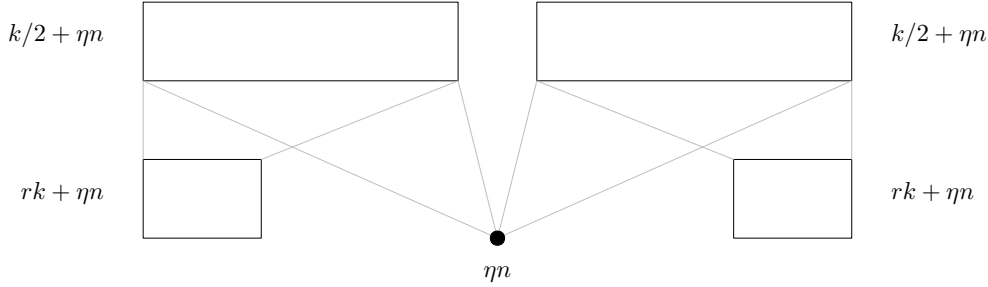}
    \caption{Example showing that the condition on bounded degree is needed in the statement of Theorem \ref{thm:localES_dense_skew}. }
    \label{fig:locales_example}
\end{figure}

\begin{proof}
Suppose $0 < r < 1/3$ and pick $\eta > 0$ to be sufficiently small depending on the value of $r$. 

Let $G$ be a graph on $n$ vertices consisting of two disjoint copies of complete bipartite graphs with colour classes of sizes $rk + \eta n$ and $k/2 + \eta n$. Moreover, $\eta n$ additional vertices are complete to both larger colour classes of the two bipartite graphs (see Figure \ref{fig:locales_example}). This implies that $k = \frac{1-5\eta}{1+2r}n$; for simplicity we do not address the rounding issues regarding $k$. 

Let $T$ be a tree on $k$ vertices consisting of a vertex $x$ complete to centres of $rk$ stars of sizes $\lfloor \frac1r \rfloor$ and $\lceil \frac1r \rceil$.

The smaller colour class of $T$ has size $rk$. Note that for fixed $r$ the maximum degree of this smaller colour class of $T$ is constant. However, it is not true for the larger colour class, hence Theorem \ref{thm:localES_dense_skew} does not apply. 
We claim that the tree $T$ is not contained in $G$ if we have chosen $\eta$ to be sufficiently small. 

Suppose that there is an embedding of $T$ in $G$. The graph $G$ is bipartite with one colour class of size at most $2rk+3\eta n$. Since $k$ is linear in $n$ and $r < \frac13$, we can choose $\eta$ small enough depending on $r$ so that this expression is less than $(1-r)k$. Hence, the vertex $x$ must be embedded in the larger colour class. Out of $(1-r)k - 1$ leaves at least $(1-r)k -1 - \eta n\cdot \lceil\frac{1}{r}\rceil > k/2 + \eta n$ have to be embedded in the same set of size $k/2 + \eta n$ as $x$, a contradiction with the assumption that $G$ embeds $T$. 
\end{proof}

Theorem \ref{thm:localES_dense_skew} is thus an example of an asymptotic result that does not seem to have a natural exact strengthening. 
On the other hand, we believe that the assumption on the sublinear maximum degree in Theorem \ref{thm:localES_dense_skew} can be dropped in the case $r=1/2$. This would mean that the dense asymptotic version of the \ES conjecture could be proven by this local approach.

\subsection{The constant $\frac12$ in Conjecture \ref{conj:dense_big_degree} cannot be improved}

The following example shows that the constant $\frac12$ in Conjecture \ref{conj:dense_big_degree} is best possible. 

Let $k > 1$ be an odd square and $T$ be a tree of order $k+1$ consisting of a vertex connected to centres of $\sqrt{k}$ stars on $\sqrt{k}$ vertices. 
Let $G$ be a graph consisting of two disjoint cliques of order $\frac{k-1}{2}$ and $\frac{k+1}{2}$, and an independent set of $\frac{\sqrt{k}-1}{2}$ vertices complete to both cliques. A simple calculation shows that the proportion of high degree vertices of $G$ is 
\[
    \frac{\frac{\sqrt{k}-1}{2}}{k+\frac{\sqrt{k}-1}{2}}
    < \frac{1}{2\sqrt{k}}\;. 
\]
On the other hand, note that for any $c<1$ the left hand side is larger than $\frac{c}{2\sqrt{k}}$ for sufficiently large $k$. We will check that $G$ does not contain $T$, which in turn shows that the expression $\frac{n}{2\sqrt{k}}$ in the conjecture cannot be strengthened to $\frac{cn}{2\sqrt{k}}$ for any $c<1$. 

Suppose that $G$ embeds $T$. If the central vertex of $T$ is embedded in the independent set, at least $\frac{\sqrt{k}+1}{2}$ stars neighbouring with that vertex have to be embedded in one of the cliques together with the independent set. 
This means that we have to embed at least $1 + \frac{\sqrt{k}+1}{2} \cdot \sqrt{k} = \frac{k}{2} + \frac{\sqrt{k}}{2} + 1$ vertices in part of $G$ consisting of at most $\frac{k+1}{2} + \frac{\sqrt{k}-1}{2} = \frac{k}{2} + \frac{\sqrt{k}}{2}$ vertices, which is not possible. 

Suppose that the central vertex of $T$ is embedded in one of the cliques. We can embed at most $\frac{\sqrt{k}-1}{2} \cdot (\sqrt{k}-1) = \frac{k-2\sqrt{k}+1}{2}$ vertices of $T$ in the other clique, since each vertex in the independent set enables us to embed $\sqrt{k}-1$ leaves in the other clique. 
We cannot use at least $\frac{k-1}{2} - \frac{k-2\sqrt{k}+1}{2} = \sqrt{k}-1$ vertices of the host graph for the embedding of $T$. Only $|V(G)| - (\sqrt{k}-1) = k + \frac{\sqrt{k}-1}{2} - (\sqrt{k} - 1) < k$ vertices in $G$ can be used to embed $T$, hence it is again not possible to embed $T$.

\section{The regularity method}
\label{sec:techniques}

In this section we state several preparatory results that will be later used for the proof of Theorem \ref{thm:localES_dense_skew}.

The basic idea of using the regularity lemma for embedding trees is that it is easy to embed trees when we know that the host graph is (pseudo)random, because then we can use its expansion properties. 
The regularity lemma (Subsection \ref{subsec:regularity}) enables us to partition the host graph into bounded number of clusters such that the edges between them are behaving in a pseudorandom fashion. We cannot exploit this property to easily embed the whole tree, but we may partition the tree into small subtrees (Subsection \ref{subsec:partitioning}) and then it is reasonably easy, though technical, to embed any such small subtree in basically any pair of clusters with nontrivial amount of edges between them (the second technical lemma in Subsection \ref{subsec:embedding_lemmas}). 

The problem that we are left with (proof in Section \ref{sec:proof}) is to embed the macroscopic structure of the tree that we get after its partitioning in the clusters of the host graph. This problem is quite similar to the problem that we started with; it is indeed tempting to think about this problem as of a fractional embedding, since we may embed several small subtrees in overlapping clusters subject to cardinality constraints. The reality is more complicated and technical, though, so we do not pursue this intuition later in the paper. 



\subsection{Notation}
\label{subsec:notation}
Throughout the paper we will use the following notation. 
The edge density of a bipartite graph with colour classes $X, Y$ is the fraction $\frac{e(X,Y)}{|X||Y|}$, where $e(X,Y) = |E(X,Y)|$ and $E(X,Y)$ is the set of edges with one endpoint in $X$ and the other in $Y$. The average degree is defined as $\adeg(X,Y) = \frac{e(X,Y)}{|X|}$. When we work with a fixed graph $G$, we use $V(G)$ and $E(G)$ to denote the set of its vertices or edges, respectively. For $X \subseteq V(G)$ we then also write $\adeg(X)$ instead of $\adeg(X, G\setminus X)$ -- this is the average degree of a vertex of $X$ in the subgraph of $G$ induced by edges between $X$ and $V(G) \setminus X$. 
The neighbourhood of a vertex $v$ is the set of vertices $N_G(v) = \{u \in G | \{u,v\} \in E(G)\}$. The neighbourhood of a set $S$ is $N_G(S) = \{u \in G | \exists v \in S: \{u,v\} \in E(G)\}$. If $T$ is a tree and $x,y \in T$, $\dist_T(x,y)$ is the length of the unique path between $x$ and $y$ in $T$, i.e., the number of edges on that path.     

\subsection{Regularity lemma}
\label{subsec:regularity}

We say that $(X,Y)$ is an {\em $\varepsilon$-regular pair}, if for every $X'\subseteq X$ and $Y'\subseteq Y$, $|X'|\geq \varepsilon |X|$ and $|Y'|\geq \varepsilon |Y|$ it holds that $|\dens(X',Y')-\dens(X,Y)| \leq \varepsilon$.


We say that a partition $\{\bf{v}_0,\bf{v}_1,\ldots, \bf{v}_m\}$ of $V(G)$ is an {\em $\varepsilon$-regular partition}, if $|\bf{v}_0|\leq \varepsilon |V(G)|$, and all but at most $\varepsilon m^2$ pairs $(\bf{v}_i,\bf{v}_j)$, $1 \le i < j \le m$, are $\varepsilon$-regular. Each set of the partition is called cluster. We call the cluster $\bf{v}_0$ the {\em garbage set}. We call a regular partition {\em equitable} if $|\bf{v}_i|=|\bf{v}_j|$ for every $1 \leq i<j\le m$.

\begin{theorem}[Szemerédi's regularity lemma]
\label{thm:regurality_lemma}
For every $\eps > 0$ there is $n_0$ and $M$ such that every graph of size at least $n_0$ admits an $\eps$-regular equitable partition $\{\bf{v}_0, \dots, \bf{v}_m\}$ with $1/\eps \le m \le M$. 
\end{theorem}

Given an $\varepsilon$-regular pair $(X,Y)$, we call a vertex $x\in X$ {\em typical} with respect to a set $Y'\subseteq Y$ if $\deg(x,Y')\ge  (\dens(X,Y)-\varepsilon)|Y'|$.
Note that from the definition of regularity it follows that all but at most $\varepsilon |X|$ vertices of $X$ are typical with respect to any subset of $Y$ of size at least $\varepsilon|Y|$.

\subsection{Partitioning trees}
\label{subsec:partitioning}

Here we state a crucial lemma from \cite{HladkyLKS4} that allows us to partition the tree in controllable number of small subtrees that we also informally call \emph{microtrees}. 
These trees are neighbouring with a set of vertices of bounded size consisting of vertices that we informally call \emph{seeds}. Moreover, we need to work separately with seeds from different colour classes of $T$. 

\begin{figure}
    \centering
    \includegraphics{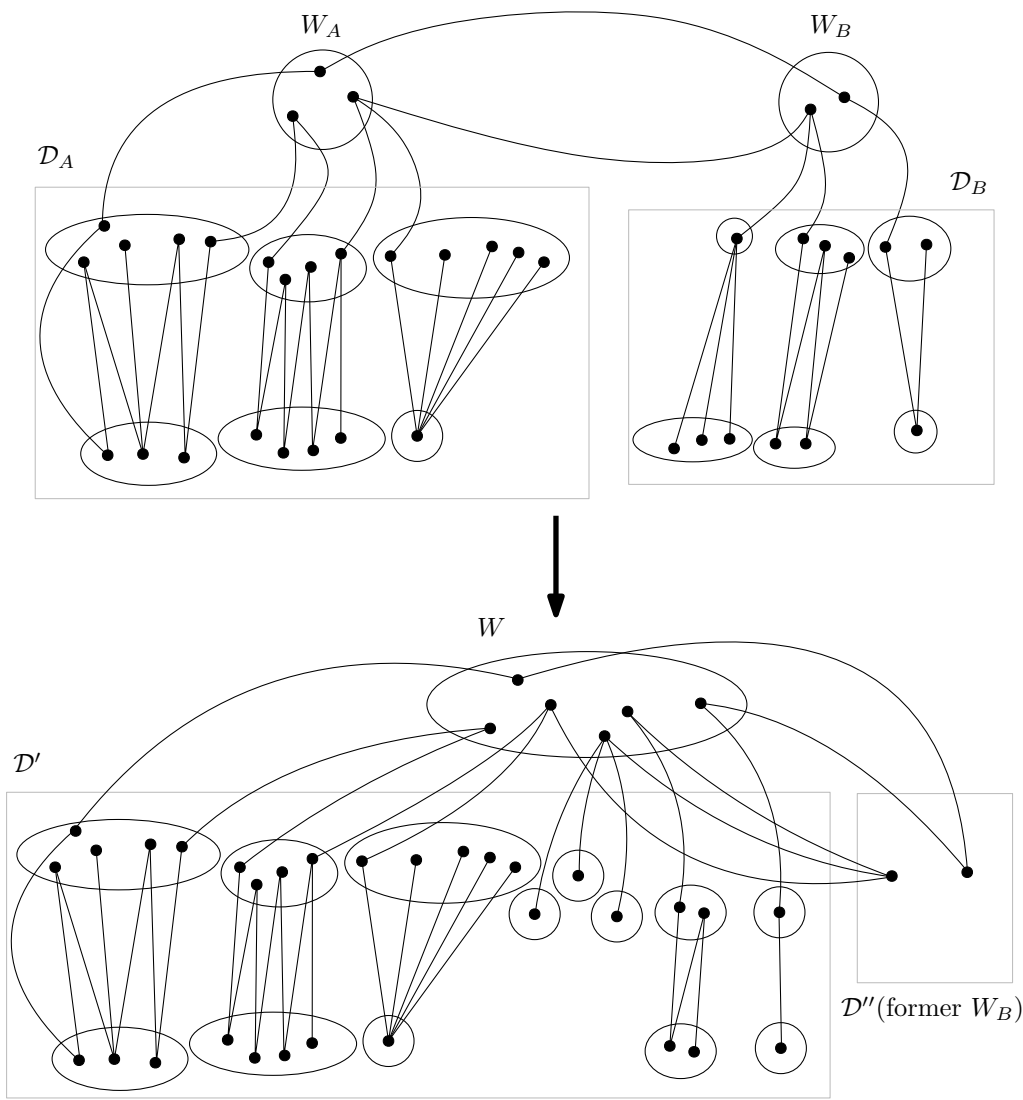}
    \caption{An $8$-fine partition of a tree and the respective one-sided $8$-fine partition of the same tree from the proof of Lemma \ref{lem:partition_bounded}. }
    \label{fig:partition}
\end{figure}

In the following definition, as well as in the example partition in the Figure \ref{fig:partition}, the set $W_A \cup W_B$ is the set of seeds of $T$ and the set $\fD_A \cup \fD_B$ is the set of its microtrees. 

\begin{definition}\cite[Definition 3.3] {HladkyLKS4}
\label{def:partition}
Let $T$ be a tree on $k+1$ vertices. An $\ell$-fine partition of~$T$ is a quadruple $(W_A,W_B, \mathcal D_A, \mathcal D_B)$, where $W_A,W_B\subseteq V(T)$ and $\mathcal D_A$ and $\mathcal D_B$ are families of subtrees of~$T$ such that
\begin{enumerate}
    \item  the sets $W_A$, $W_B$ and $\{V(K)\}_{K\in\mathcal D_A\cup
    \mathcal D_B}$ partition $V(T)$ (in particular, the trees in $K\in\mathcal D_A\cup
    \mathcal D_B$ are pairwise vertex disjoint),
    
    \item $\max\{|W_A|,|W_B|\}\leq 336k/{\ell}$,
    
    \item for $w_1,w_2\in W_A\cup W_B$ their distance in $T$ is odd if and only if one of them lies in $W_A$ and the other one in $W_B$,
    
    \item $|K| \leq \ell$ for every tree $K\in \mathcal D_A\cup
    \mathcal D_B$,
    
    \item for each $K \in \mathcal D_A$ we have $N_T(V(K)) \setminus V(K) \subseteq W_A$. Similarly for each $K \in \mathcal D_B$ we have $N_T(V(K)) \setminus V(K) \subseteq W_B$.  
    
    \item  $|N_T(V(K))\cap (W_A\cup W_B)|\le 2$ for each $K\in\mathcal D_A\cup \mathcal D_B$,
    
    \item if $N_T(V(K)) \cap (W_A \cup W_B)$ contains two vertices $z_1$, $z_2$ for some $K \in \fD_A \cup \fD_B$, then $\dist_T(z_1, z_2) \ge 6$. 
\end{enumerate}
\end{definition}

We did not list all properties of $\ell$-fine partition from~\cite{HladkyLKS4}, only those we need. 

\begin{lemma}\cite[Lemma~3.5]{HladkyLKS4}
    \label{lem:partition}
 	Let $T$ be a tree on $k+1$ vertices and let $\ell\in \mathbb{N}, \ell<k$. Then $T$ has an $\ell$-fine partition.
\end{lemma}

In the subsequent applications we are always working with $\ell = \beta k$ for some small $\beta > 0$. 
%
%
%
%
%

Since we work with trees with sublinear degree, we may further constrain the $\ell$-fine partition in such a way that all of its seeds are only from one colour class of $T$. We call this simpler structure a \emph{one-sided} $\ell$-fine partition. 

We at first define precisely the notion of a one-sided $\ell$-fine partition and in the subsequent lemma we observe that we may get the one-sided $\ell$-fine partition from $\ell$-fine partition by adding neighbours of seeds of $W_B$ to $W_A$ and then treating the former set $W_B$ similarly to the set $\fD_A \cup \fD_B$ (the lower part of Figure \ref{fig:partition}).

\begin{definition}
\label{def:partition_bounded}
Let $T$ be a tree on $k+1$ vertices and $V_1,V_2$ its colour classes. Let $\Delta = \max_{v \in V_2} \deg(v)$. A one-sided $\ell$-fine partition of~$T$ is a pair $(W, \mathcal D)$, where $W \subseteq V(V_1)$ and $\mathcal D$ is a family of subtrees of~$T$ such that
\begin{enumerate}
    \item  the sets $W$ and $\{V(K)\}_{K\in\mathcal D}$ partition $V(T)$,
    
    \item $|W|\leq 336k(1+\Delta)/{\ell}$,
    
    \item $|K| \leq \ell$ for every tree $K\in \mathcal D$,
    
    \item For each $K \in \mathcal D$ we have $N_T(V(K)) \setminus V(K) \subseteq W$. 
    
    \item 
    We can split $\fD$ into two subfamilies, $\fD = \fD' \sqcup \fD''$, in such a way that all trees from $\fD'$ have at most two neighbours $z_1, z_2 \in W$ such that $\dist_T(z_1, z_2) \ge 4$, while $|\fD''| \le 336k/\ell$ and every tree from $\fD''$ is a single vertex with at most $\Delta$ neighbours in $W$.  
\end{enumerate}
\end{definition}

\begin{lemma}
\label{lem:partition_bounded}
Let $T$ be a tree on $k+1$ vertices and let $\ell\in \mathbb{N}, \ell<k$. Then $T$ has a one-sided $\ell$-fine partition.
\end{lemma}

\begin{proof}
Let $(W_A, W_B, \fD_A, \fD_B)$ be an $\ell$-fine partition of $T$. Let $V_1, V_2$ be the partition of the vertices of $T$ into colour classes. Suppose that $W_B \subseteq V_2$. Let $W = W_A \cup N_T(W_B)$ and define $\fD$ as the set of trees of the forest $T \setminus W$. 
The conditions (1), (2), and (4) are clearly satisfied. 
Each vertex from $W_B$ is now a singleton tree in $\fD$. Define $\fD''$ as the family of these singleton trees and set $\fD' = \fD \setminus \fD''$. 
Each tree in $\fD''$ clearly satisfies the conditions (3) and (5). 
Each tree from $\fD'$ is either a tree from $\fD_A$, or a subtree of a tree from $\fD_B$, all such trees satisfy the condition (3). 
Finally recall that for each tree from $\fD_A \cup \fD_B$ with two neighbours $z_1$ and $z_2$ in $W_A \cup W_B$ we have $\dist_T(z_1, z_2) \ge 6$. Thus, all trees from $\fD_A$ satisfy the condition (5). Each tree from $\fD_B$ with two neighbours $z_1, z_2 \in W_B$ was split into one tree with two neighbours in $W$, such that their distance in $T$ is at least $4$, and maybe several other trees with only one neighbour in $W$. All such trees also satisfy (5). 
\end{proof}

\subsection{Embedding in regular pairs}
\label{subsec:embedding_lemmas}
In this section we present two embedding lemmas. The first will be used to embed the seeds of a one-sided partition, together with the set $\fD''$, in vertices of two neighbouring clusters. 

\begin{proposition}
\label{prop:embed_seeds}
For any $d, \beta, \eps > 0$, $\eps \le d^2/100$ there exist $k_0$ and $\gamma > 0$ such that the following holds. 

Let $T$ be a tree of order $k \ge k_0$ and $V_2$ one of its colour classes such that $\forall v \in V_2: \deg(v) \le \gamma k$. Moreover, let $(W,\fD), \fD=\fD'\sqcup\fD''$ be its one-sided $\beta k$-fine partition.  
Let $\bf{v}_1$ and $\bf{v_2}$ be two clusters of vertices of a graph $G$ forming an $\eps$-regular pair of density at least $d$. 
Suppose that $|\bf{v}_1| = |\bf{v}_2| \ge k / M_{T\ref{thm:regurality_lemma}}(\eps)$, where $M_\reg(\eps)$ is the output of the regularity lemma (Theorem \ref{thm:regurality_lemma}) with an input $\eps$. Let $U \subseteq \bf{v}_1, |U| \leq 2\sqrt{\eps}|\bf{v}_1|$. 
Then there is an injective mapping $\phi$ of $W \cup \left( \bigcup \fD'' \right)$ that embeds vertices of $W$ in $\bf{v}_1 \setminus U$ and vertices of $\bigcup \fD''$ in $\bf{v}_2$. 
\end{proposition}

\begin{proof}
Choose $\gamma, k_0 > 0$ such that 

\begin{align*}
    \gamma &= \frac{\beta d}{2000 M_{\reg}(\eps)},\\
    k_0 &= \frac{10}{\gamma}. 
\end{align*}

Note that in this case we have 

\begin{align*}
     \left\vert \bigcup D''\right\vert + |W|
     &\le \frac{336k}{\beta k} + \frac{336k(1+\gamma k)}{\beta k}\\
    &=\frac{336(\gamma k + 2))}{\beta}\\
    \JUSTIFY{$k \ge 10/\gamma$}&\le\frac{500\gamma k}{\beta}\\
    \JUSTIFY{definition of $\gamma$} &= \frac{500 \beta d k}{\beta \cdot  2000 M_{\reg}(\eps)} 
    = \frac{d k}{4M_{\reg}(\eps)} \\
    \JUSTIFY{$|\bf{v}_1| \ge k / M_{\reg}(\eps)$} &\le \frac{d}{4} | \bf{v}_1 | .\\
\end{align*}

Take an arbitrary vertex $r \not\in \bigcup\fD''$ of $T$ and root the tree at $r$. Order all vertices of $W \cup \left( \bigcup\fD''\right)$ according to an order, in which they are visited by a depth-first search starting at $r$. Let $U' \subseteq \bf{v}_1 \cup \bf{v}_2$ be the set of vertices of $\bf{v}_1$ not typical to $\bf{v}_2$ together with vertices of $\bf{v}_2$ not typical to $\bf{v}_1$. 
We will provide an algorithm that gradually defines a partial embedding $\phi$ of the vertices of $W \cup \left( \bigcup\fD''\right)$ such that $\phi(W) \subseteq \bf{v}_1 \setminus (U \cup U')$ and $\phi(\bigcup \fD'') \subseteq \bf{v}_2 \setminus U'$.  The situation is displayed in Figure \ref{fig:emb_lemma1}. 

\begin{figure}
    \centering
    \includegraphics{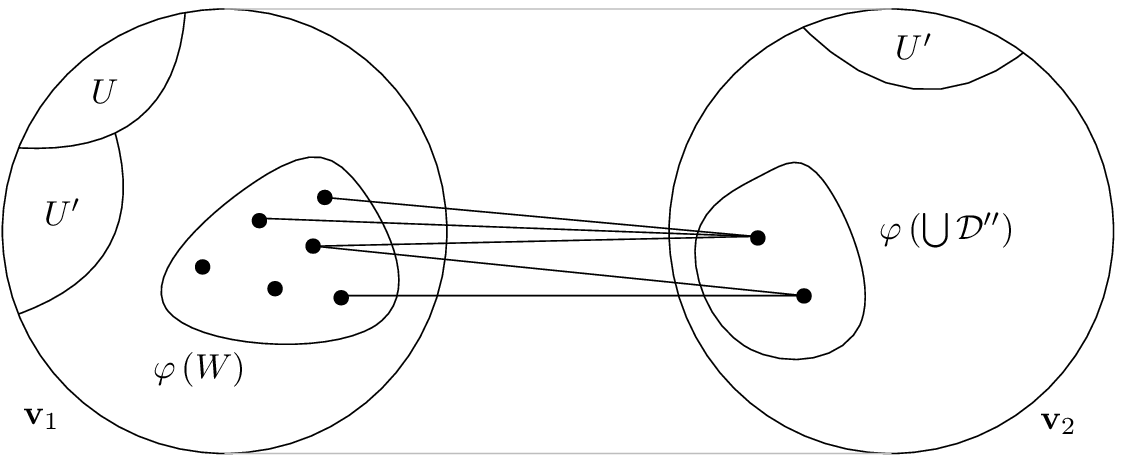}
    \caption{Embedding of $W$ and $\bigcup \fD''$ in the two clusters $\bf{v}_1$ and $\bf{v}_2$ in Proposition \ref{prop:embed_seeds}. }
    \label{fig:emb_lemma1}
\end{figure}

We iterate over the sequence $x_1, x_2, x_3, \dots$ of vertices from $W \cup \left( \bigcup\fD''\right)$, where the vertices are ordered by the depth-first search. In the $i$-th step we deal with the vertex $x=x_i$.
First we deal with the case $x \in W$. 

Suppose that $y \in \bigcup \fD''$ is the already embedded parent of $x$ (if $y \not \in \bigcup \fD''$, our task is simpler, since we do not have to embed $x$ in the neighbourhood of $\phi(y)$). 
We want to embed $x$ in an arbitrary neighbour of $y$ in $\bf{v}_1 \setminus (U \cup \phi(W) \cup U' )$. To do so, it suffices to verify that $N_G(y) \setminus (U \cup \phi(W) \cup U')$ is nonempty. This can be done with the help of the fact that $\phi(y)$ is typical to $\bf{v}_1$ and together with our bound $|W| \le \frac{d}{4}|\bf{v}_1|$:
\begin{align*}
 | N_G(y) \setminus (U \cup \phi(W) \cup U') | 
 \ge |\bf{v}_1| ( (d-\eps)  - 2\sqrt{\eps}  - \frac{d}{4} - \eps )
 > 0.
\end{align*}

Similarly, suppose that $x \in \bigcup \fD''$. From the definition of $\fD''$ we know that its parent $y$ is certainly in $W$ and $\phi(y)$ is typical to $\bf{v}_2$.  Now we similarly verify that 
\begin{align*}
 \left\vert N_G(y) \setminus \left(\phi \left(\bigcup \fD''\right) \cup U' \right) \right\vert 
 \ge |\bf{v}_2| ( (d-\eps)  - \frac{d}{4} - \eps )
 > 0.
\end{align*}

\end{proof}

Next, we state a similar proposition that enables us to embed small trees from a fine partition of $T$ in the regular pairs of the host graph. The proposition is a variation on a folklore result and is similar to e.g. Lemma 5 in \cite{Klimosova2018+}. 

\begin{proposition}
\label{prop:embed_regular_pair}
For all $0 < d, \eps \le 1$ such that $\eps < d^2 / 100$ there exists $\beta > 0$ such that the following holds. 

Let $\bf{v}_1, \bf{u}, \bf{v}$ be three clusters of vertices of a graph $G$ such that $\bf{v}_1\bf{u}$ and $\bf{uv}$ are $\eps$-regular pairs of density at least $d$. Let $v_1, v_2$ be two (not necessarily distinct) vertices of $\bf{v}_1$. 
Suppose that $|\bf{v}_1| = |\bf{u}| = |\bf{v}| \ge k / M_{\reg}(\eps)$. 
Let $K$ be a tree of order at most $\beta k$ and let $x_1, x_2$ be any of its two (not necessarily distinct) vertices from the same colour class of $K$. Suppose that $v_1 \not = v_2$ if and only if $x_1 \not = x_2$. 
Let $U$ be any subset of vertices of $\bf{u} \cup \bf{v}$ such that $|\bf{u} \setminus U| \ge 4 \sqrt{\eps} |\bf{u}|$ and $|\bf{v} \setminus U| \ge 4 \sqrt{\eps} |\bf{v}|$. 
Suppose that $|N_G(v_i)\cap (\bf{u}\setminus U)| \ge 3\eps|\bf{u}|$ for $i=1,2$. 


Then there is an injective mapping $\phi$ of $K$ in $\bf{u}\cup \bf{v}$ such that $\phi(V(K)) \cap U = \emptyset$. Moreover, $\phi(x_1) \in \bf{u}$ is a neighbour of $v_1$ and $\phi(x_2) \in \bf{u}$ is a neighbour of $v_2$. 
\end{proposition}

\begin{proof}
We show the proof for the harder case when $v_1 \not = v_2$. 
Choose 
\begin{align*}
\beta = \eps / M_\reg(\eps). 
\end{align*}
From this we get 
\begin{align*}
    |\bf{v}_1| 
    \ge \frac{k}{M_{\reg} (\eps)} = \beta \cdot \frac{M_\reg(\eps)}{\eps} \cdot  \frac{k}{M_{\reg} (\eps)} = \frac{\beta k}{\eps}.  
\end{align*}

Let $U'$ be the set of at most $\epsilon|\bf{u}|$ vertices of $\bf{u}$ that are not typical to $\bf{v} \setminus U$ (note that $|\bf{v} \setminus U| \ge \epsilon|\bf{v}|$) together with at most $\epsilon|\bf{v}|$ vertices of $\bf{v}$ that are not typical to $\bf{u} \setminus U$. The situation is displayed in Figure \ref{fig:emb_lemma2}. 

\begin{figure}
    \centering
    \includegraphics{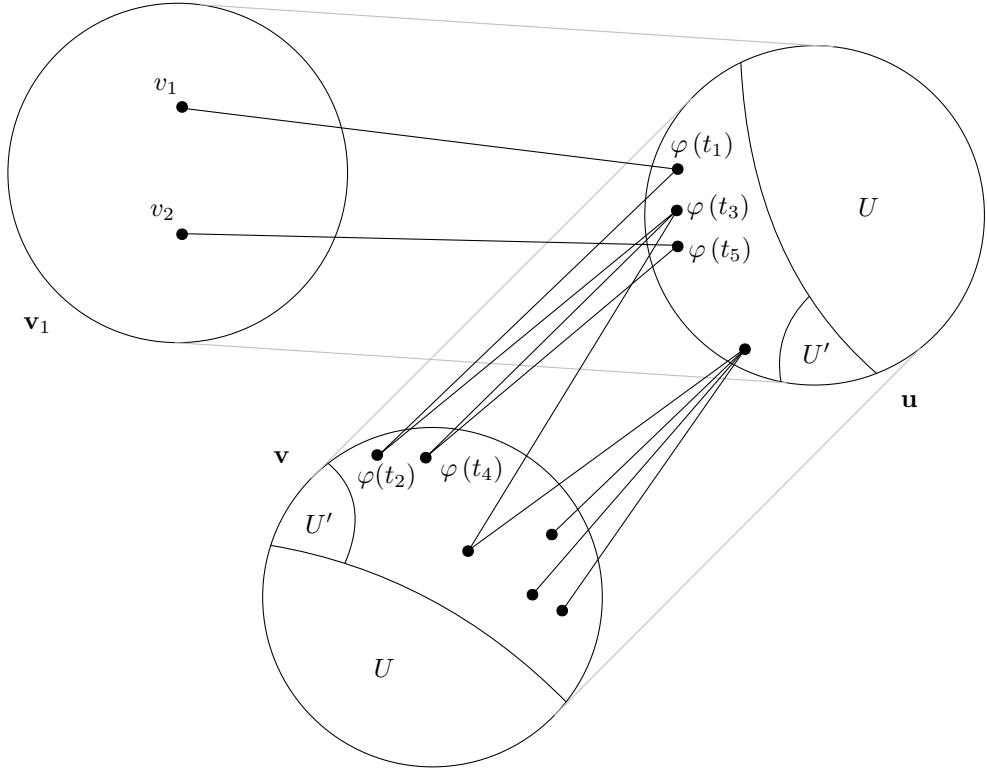}
    \caption{Embedding of a small tree $K$ in two clusters $\bf{u}$ and $\bf{v}$ in Proposition \ref{prop:embed_regular_pair}. In this case $r = 5$. }
    \label{fig:emb_lemma2}
\end{figure}

Observe that for each vertex $u \in \bf{u} \setminus (U \cup U')$ we have
\begin{align*}
    |N_G(u)\cap \left(\bf{v}\setminus (U \cup U')\right)|
    &\ge (d-\eps)|\bf{v}\setminus U|-\eps|\bf{v}|\\ 
    \JUSTIFY{$|\bf{v}\setminus U| \ge 4\sqrt{\eps}|\bf{v}|$}&\ge (d-\eps)4 \sqrt{\eps}|\bf{v}| - \eps|\bf{v}|\\
    \JUSTIFY{$d\gg \sqrt{\eps}$}&\ge \sqrt{\eps}\cdot 4 \sqrt{\eps} |\bf{v}| - \eps|\bf{v}| 
    \ge 2\eps |\bf{v}| \\
    \JUSTIFY{$|\bf{v}| \ge \beta k / \eps$}&\ge \eps|\bf{v}| + \beta k
    \ge \eps|\bf{v}| + |K| ,    
\end{align*}
and similar holds for any $u \in \bf{v} \setminus (U \cup U')$. This means that during embedding we may always find a neighbour of $u$ in $\bf{v} \setminus (U \cup U')$ that was not yet used for embedding. 

The same applies for both vertices $v_1, v_2$. We have
\begin{align*}
    |N_G(v_i) \cap \left(\bf{u}\setminus (U \cup U')\right)|
    &\ge |N_G(v_i) \cap (\bf{u}\setminus U)| - \eps |\bf{u}|\\
    \JUSTIFY{$|N_G(v_i)\cap (\bf{u}\setminus U)| \ge 3\eps|\bf{u}|$}
    &\ge 2\eps|\bf{u}|\\
    \JUSTIFY{$|\bf{u}| \ge \beta k / \eps$}&\ge \eps|\bf{u}|+\beta k
    \ge \eps|\bf{u}|+|K|.
\end{align*}

We start by embedding the path $t_1=x_1, t_2, \dots, t_r=x_2$ connecting $x_1$ with $x_2$ in $K$. We embed these vertices alternately in clusters $\bf{u}$ and $\bf{v}$. We embed $x_1$ in an arbitrary vertex of $N_G(v_1) \cap \left(\bf{u} \setminus (U \cup U')\right)$. 
Now for $i$ going from $2$ to $\ell-2$ we always map $t_{i}$ to a neighbour of $\phi(t_{i-1})$ not lying in $U \cup U'$. Observe that both $N_G(v_2) \cap \left(\bf{u}\setminus (U \cup U')\right)$ and $N_G(t_{r-2}) \cap \left( \bf{v}\setminus ( U \cup U')\right)$ have sizes at least $\eps|\bf{v}_1|$, thus there is an edge connecting those two neighbourhoods. Map $t_{r-1}$ and $t_{r}$ to the two endpoints of the edge. The rest of the tree can be then embedded in the greedy manner. 
\end{proof}

\section{Proof of Theorem \ref{thm:localES_dense_skew}}
\label{sec:proof}

In this section we prove Theorem \ref{thm:localES_dense_skew}. We split the proof into three parts. At first we preprocess the host graph by applying the regularity lemma and we partition the tree by applying Lemma~\ref{lem:partition_bounded}. In the second part we find a suitable matching structure in the host graph. In the last part we embed the tree in the host graph. 

\subsection*{Preprocessing the host graph and the tree}

Fix $\eta, r$. Suppose that $\eta < 1$. Choose $d, \eps, \beta, n_0$ such that
\begin{align*}
    d &= \frac{(\eta r)^2}{1000} ,\\
    \eps &=  \frac{(\eta r d)^{20}}{10^{15}}, \\
    \beta &= \min\left( \beta_{P\ref{prop:embed_regular_pair}}(d, \eps), \frac{\eta d}{10^{5}\cdot M_{\reg}(\eps)} \right),\\
    \gamma &= \gamma_{P\ref{prop:embed_seeds}}(d, \eps, \beta),\\
    n_0 &= \max\left(n_{0,\reg}(\eps), \frac{2}{\eta} k_{0, P\ref{prop:embed_seeds}}(d, \eps, \beta) \right).
\end{align*}

Let $G$ be a fixed graph on $n \ge n_0$ vertices with at least $\eta n$ vertices of degree at least $k + \eta n$ and with $\delta(G) \ge rk + \eta n$. Suppose that $k \ge \eta n/2$, otherwise we embed the tree $T$ greedily. 
We apply the regularity lemma (Theorem~\ref{thm:regurality_lemma}) on $G$ with $\eps_{\reg} = \eps$ and obtain an $\eps$-regular equitable partition $\bf{v}_0, \bf{v}_1, \dots, \bf{v}_m$ with $1/\eps \le m \le M_{\reg}(\eps)$ clusters. Each cluster has average degree at least $rk + \eta n$. 

Erase all edges within sets $\bf{v}_i$ of the partition, between irregular pairs, and between pairs of density less than $d$. 
We have erased at most $m\cdot \binom{n/m}{2} \le \frac{n^2}{m} \le \eps n^2$ edges withing the sets $\bf{v}_i$, at most $\eps m^2 \cdot (n/m)^2 = \eps n^2$ edges in irregular pairs, and at most $\binom{m}{2} \cdot d \cdot (n/m)^2 \le d \cdot n^2$ edges in pairs of low density. Erase the garbage set $\bf{v}_0$ and all of the at most $\eps n \cdot n$ incident edges.  
Note that we have erased at most $(3\eps + d)n^2$ edges. We abuse notation and still call the resulting graph $G$. 

Note that the quantity $\sum_{1 \le i \le m} |\bf{v}_i| \cdot \adeg(\bf{v}_i)$ dropped down by at most $(6\eps +2d)n^2$. Thus there are at most $\sqrt{6\eps + 2d}\cdot m$ clusters such that their average degree dropped down by more than $\sqrt{6\eps+2d}\cdot n$. Delete all such clusters and incident edges. We again call the resulting graph $G$. 
The average degree of each cluster of $G$ that was not deleted at first dropped by at most $\sqrt{6\eps+2d}\cdot n$. Then we erased at most $\sqrt{6\eps+2d} \cdot m$ clusters, so now it is at least $rk + \eta n - 2\cdot\sqrt{6\eps+2d} \cdot n > rk + \eta n/2$. 
Moreover, $G$ contains at least $(\eta - \eps - \sqrt{6\eps+2d}) n \ge \eta n/2$ vertices of degree at least $k + \eta n - 2\cdot\sqrt{6\eps+2d}\cdot n \ge k + \eta n/2 $. Hence, there exists a cluster, without loss of generality it is $\bf{v}_1$, such that the proportion of vertices of degree at least $k + \eta n/2$ in that cluster is at least $\eta /2 \ge \epsilon$. If we denote by $L$ this set of high degree vertices of $\bf{v}_1$, we have $\adeg(\bf{v}_1, \bf{v}_i) \ge \adeg(L, \bf{v}_i) - \eps |\bf{v}_i|$ from regularity of each pair $(\bf{v}_1, \bf{v}_i)$. This yields that $\adeg(\bf{v}_1) \ge \adeg(L) - \eps n \ge k + \eta n /3$. Moreover, if it is the case that $\adeg(\bf{v}_1) > 2k$, we erase several regular pairs with one endpoint in $\bf{v}_1$ so as to achieve $\adeg(\bf{v}_1) \le 2k$. After deletion the average degree of each cluster is still at least $k + \eta n /2 - n/m \ge k + \eta n /3$. 

The cluster graph $\bf G$ of $G$ is a graph such that its vertex set are the clusters of $G$ and between any pair of vertices $\bf{u}, \bf{v} \in \bf{G}$ there is an edge with weight $d(\bf{u}, \bf{v})$ if and only if $\bf{uv}$ is a regular pair with density $d(\bf{u}, \bf{v}) > 0$. 
Since we already deleted all irregular pairs and pairs with low density, for any edge in the cluster graph we have $d(\bf{u}, \bf{v}) \ge d$. 

We stick to using the boldface font whenever we may think about the corresponding object as a vertex or a set of vertices of $\bf G$.
We use $N_{\bf{G}}(\bf{v})$ to denote the set of clusters of $\bf{G}$ that are neighbours of $\bf{v}$ in $\bf{G}$, while $N_G(v)$ denotes the neighbourhood of a vertex $v$ of $G$. 

After preprocessing the host graph we turn our attention to the tree $T$.  Let $V_1, V_2$ be its colour classes such that $|V_1| \le rk$ and $\forall v \in V_2 : \deg(v) \le \gamma k$. We apply Lemma~\ref{lem:partition_bounded} with parameter $\ell_{L\ref{lem:partition_bounded}} = \beta k$ and obtain its one-sided $\beta k$-fine partition $(W,\fD), \fD=\fD'\sqcup\fD''$ such that $|W| \le 336(1+\gamma k)/\beta$ and $|\bigcup \fD''| \le 336/\beta$. Moreover, for each $K \in \fD'$ we have $|K| \le \beta k$ and for each $K \in \fD''$ we have $|K|=1$. 
Also note that $W \subseteq V_1$. 


\subsection*{Structure of the host graph}

We now find a suitable structure in the cluster graph $\mathbf{G}$ that will be used for the embedding of $T$. It suffices to look at the cluster $\mathbf{v}_1$, that will serve for the embedding of the seeds of $T$, and its neighbourhood. 

Let $\bf{M}$ a maximal matching in $N_{\bf{G}}(\bf{v}_1)$. We will denote by $\bf{M}$ both the graph and its underlying vertex set. 
Suppose that $\bf{uv} \in \bf{M}$. Note that from the condition on maximality we get that there cannot be two vertices $\bf{x} \not= \bf{y} \in N_{\bf{G}}(\bf{v}_1) \setminus \bf{M}$ such that both $\bf{xu}$ and $\bf{yv}$ are edges of $\bf G$. Thus there are two possibilities for each edge $\bf{uv}$; either only one of its endpoints has neighbours in $N_{\bf{G}}(\bf{v}_1) \setminus \bf{M}$, or both of its endpoints have just one neighbour in $N_{\bf{G}}(\bf{v}_1) \setminus \bf{M}$. We can get rid of the second special case as follows. 
For each vertex in $N_{\bf{G}}(\bf{v}_1) \setminus \bf{M}$ we either delete it if it is a common neighbour of at least $\eta m/40$ matching pairs, or we delete all edges in at most $2\cdot \eta m  / 40$ regular pairs connecting the vertex with these matching pairs. 
In this way we delete at most $40/\eta$ clusters and the degree of all remaining clusters of $\bf G$ drops down by at most $\eta m/20 \cdot |\bf{v}_1| + 40/\eta \cdot |\bf{v}_1| \le (\eta /20+40\eps/\eta)\cdot n \le \eta n/10$. 
We abuse notation and still call the resulting graph $\bf G$. The degree of $\bf{v}_1$ is at least $k+\eta n/3 - \eta n/10 \ge k + \eta n /5$ and the average degree of every cluster is similarly at least $rk + \eta n/5$. 
The matching $\bf{M}$ is still maximal in $N_{\bf{G}}(\bf{v}_1)$. 
Moreover, we can split the vertices of $\bf M$ into two colour classes, $\bf{M} = \bf{ M}_1 \cup \bf{M}_2$, in such a way that only clusters from $\bf{M}_2$ have neighbours in $N_{\bf{G}}(\bf{v}_1) \setminus \bf M$. 
Let $\bf{O}_1 = N_{\bf{G}}(\bf{v}_1) \setminus \bf{M}$. Note that it is an independent set. Define $\bf{O}_2 = N_{\bf{G}}(\bf{O}_1) \setminus \{ \{\bf{v}_1\} \cup \bf{M} \}$. Note that $N_{\bf{G}}(\bf{v}_1) \cap \bf{O}_2 = \emptyset$. Set $\bf{O} = \bf{O}_1 \cup \bf{O}_2$. All these sets are displayed in Figure \ref{fig:embedding}. 

\begin{figure}
    \centering
    \includegraphics[width=.8\textwidth]{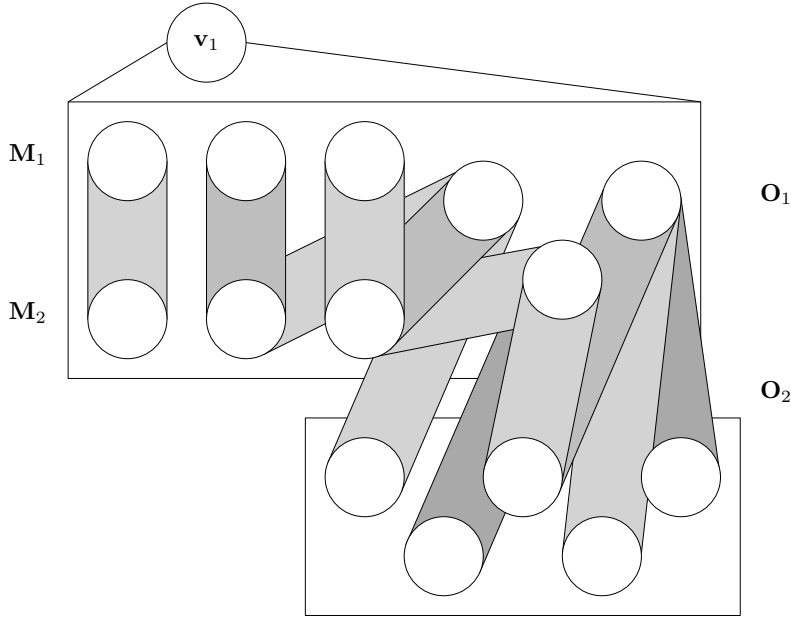}
    \caption{Cluster $\bf{v}_1$ and four sets of clusters $\bf{M}_1, \bf{M}_2, \bf{O}_1, \bf{O}_2$ that will be used for embedding. The regular pairs of different density are sketch by shades of grey (we omit pairs touching $\bf{v}_1$). }
    \label{fig:embedding}
\end{figure}
\subsection*{Embedding}
The final step of the proof is broken into three subparts. At first we give an overview of the method that we use for the construction of the mapping $\phi$. Then we formulate several preparatory technical claims. In the last part we propose the embedding algorithm. 

\subsubsection*{Overview}
We gradually construct an injective mapping $\phi$ from $T$ to $G$. In each step $\phi$ denotes the partial embedding that we already constructed. 
The idea behind the embedding process is very straightforward -- we will try to embed microtrees of $\fD$ inside the regular pairs in $\bf{M}$ and ‘through' the vertices of $\bf{O}_1$. We will, however, have to overcome several technical difficulties. 

One of the standard approaches of embedding trees (e.g. pursued in \cite{Klimosova2018+}) is to start by embedding the seeds of $T$ in vertices of two clusters (one for each colour class) such that the neighbourhoods of these special clusters are sufficiently rich. Moreover, we embed the seeds in such vertices that are typical to almost all neighbouring clusters. We then split the microtrees in $T$ into several subsets and embed each subset of microtrees in some part of the neighbourhood of the special clusters. 
Here we take a different approach. We start in the same way by embedding the seeds $W$ of $T$ in $\bf{v}_1$ -- a high degree cluster of $G$. We then propose an algorithm that iterates over clusters in the neighbourhood of $\bf{v}_1$, each time finding two clusters that can be used to embed a microtree. 

There are two main technical difficulties that we have to overcome. Recall that each seed is embedded in a vertex that is typical to almost all clusters. This means that when we choose a pair of clusters that will be used for the embedding, we have to find a microtree that has not yet been embedded such that its adjacent seeds are embedded in vertices typical to the first cluster from the pair. 
We can ensure that there will be such a microtree, unless the number of vertices that remain to be embedded, is very small, specifically $\sqrt[4]{\eps}k$. To ensure that we can embed the whole tree $T$, we at first allocate a small fraction of vertices $F \subseteq \bigcup \left( \bf{M} \cup \bf{O} \right)$ that we do not use for the embedding during the main embedding procedure. When only at most $\sqrt[4]{\eps}k$ vertices remain to be embedded, we finally embed this small proportion of trees in the set $F$. 

The second technical problem is that we cannot ensure that all the microtrees have the same skew. This complicates the main embedding procedure that would have been simpler in the case of microtrees with uniform skew. During the embedding procedure we behave against intuition and sometimes redefine the embedding of some microtrees. 


\subsubsection*{Preparations}

Note that there are at most $\sqrt{\eps} |\bf{v}_1|$ vertices of $\bf{v}_1$ that are not typical to more than $\sqrt{\eps}m$ clusters. 
Indeed, otherwise there would be more than $\sqrt{\eps} |\bf{v}_1| \cdot \sqrt{\eps}m = \eps m|\bf{v}_1|$ pairs of a cluster and a vertex not typical to it, which in turn implies the existence of a cluster such that more than $\eps |\bf{v}_1|$ vertices are not typical to it, a contradiction to the $\eps$-regularity. 
For each cluster $\bf{v} \in \bf{M}_1 \cup \bf{O}_1$ fix an arbitrary subset $F_{\bf{v}}$ of size $\lfloor \eta r d|\bf{v}|/300 \rfloor$. Since $|F_{\bf{v}}| \ge \epsilon |\bf{v}|$, we may apply the same reasoning to get that there are at most $\sqrt{\eps} |\bf{v}_1|$ vertices of $\bf{v}_1$ that are not typical to more than $\sqrt{\eps} m$ sets $\{ F_{\bf{v}_i} \}_{[2,m]}$.

We invoke Proposition~\ref{prop:embed_seeds} with parameters $d_{P\ref{prop:embed_seeds}} = d$, $\beta_{P\ref{prop:embed_seeds}} = \beta$, $\eps_{P\ref{prop:embed_seeds}} = \eps$. We also choose $\bf{v}_{2,{P\ref{prop:embed_seeds}}} = \bf{v}_2$ to be any cluster from the neighbourhood of $\bf{v}_{1,{P\ref{prop:embed_seeds}}} = \bf{v}_1$. 
Finally, we define the set $U_{P\ref{prop:embed_seeds}}$ to be the set of the at most $2\sqrt{\eps} |\bf{v}_1|$ vertices not typical to more than $\sqrt{\eps}m$ neighbouring clusters $\bf{v}_i$, or their subsets $F_{\bf{v_i}}$. 
Note that due to our initial choice of constants all the conditions from the statement of the proposition are satisfied. Hence we embed the vertices of $W$ in $\bf{v}_1$, while the vertices of $\bigcup \fD''$ will be embedded in $\bf{v}_2$. Moreover, each vertex from $W$ is embedded in a vertex typical to all but at most $\sqrt{\eps} m$ clusters $\bf{v}_i$ and their fixed subsets $F_{\bf{v}_i}$ of size $\lfloor \eta r d|\bf{v}_1|/300 \rfloor$.

Note that each microtree $K \in \fD'$ has at most two neighbours in $W$ (point five in Definition \ref{def:partition_bounded}). We call a cluster $\bf{u} \not= \bf{v}_1$ \emph{nice} with respect to $K \in \fD'$, if the at most two neighbours of $K$ in $W$ are embedded in vertices of $\bf{v}_1$ typical to $\bf{u}$. 
We will now, yet again, employ a doublecounting argument to observe that most of the clusters are nice to most of the trees from $\fD'$. 
We claim that there are at most $2\sqrt[4]{\eps} m$ clusters such that if we take all trees such that the cluster is not nice to them, then the union of all such trees contains more than $\sqrt[4]{\eps} k$ vertices. 

Suppose that it is not true. 
Note that each vertex from $W$ was mapped to a vertex that is typical to all but at most $\sqrt{\eps} m$ clusters, thus for each tree $K$ there are at most $2\sqrt{\eps} m$ clusters that are not nice to $K$. 
Consider all pairs consisting of a microtree from $\fD'$ and a cluster that is not nice to the tree. Moreover, each such connection shall be weighted by the size of the tree. Consider sum of all weights of all such pairs. 
Each tree $K$ contributes to the overall sum by at most $|K| \cdot 2\sqrt{\eps}m$, thus the overall sum is at most $k \cdot 2\sqrt{\eps}m$. On the other hand, if our claim were not true, the overall sum would be bigger than $2\sqrt[4]{\eps} m \cdot \sqrt[4]{\eps} k = k \cdot 2\sqrt{\eps}m$. 

Delete all clusters such that the union of all trees not nice to them has more than $\sqrt[4]{\eps} k$ vertices. If they are from $\bf{M}$, delete also their neighbours in $\bf{M}$. For simplicity also delete the cluster $\bf{v}_2$, because then we will not need to consider $\fD''$ in further calculations. 

Observe that the average degree of each cluster is still at least 
\begin{align*}
    rk& + \eta n/10 - 4\sqrt[4]{\eps} m  |\bf{v_1}| - |\bf{v_2}|\\
    \JUSTIFY{$m \ge 1/\eps$}&\ge rk + \eta n/10 - 4\sqrt[4]{\eps}n - \eps n\\ 
    \JUSTIFY{$\eps \ll \eta$}&\ge rk + \eta n/20. 
\end{align*} 
Similarly, the degree of $\bf{v}_1$ is still at least $\adeg(\bf{v}_1) \ge k + \eta n / 20$. We still call the new graph $\bf{G}$.
For each $\bf{u} \in N_{\bf{G}}(\bf{v}_1)$ it now holds that the number of vertices in microtrees such that $\bf{u}$ is not nice to them is at most $\sqrt[4]{\eps} k$. 

Our main embedding algorithm will work until less than $\sqrt[4]{\eps} k$ vertices of $T$ remain to be embedded. To embed the rest of the vertices of $T$, we now define a set $F$ that intersects each cluster in $\bf{M} \cup \bf{O}$ in a small fraction of vertices. 


\begin{claim}
\label{claim:localES_finishing}
There is a set $F\subseteq \bigcup \left( \bf{M} \cup \bf{O} \right)$ satisfying $|F| \le \eta r \adeg(\bf{v}_1)/100$, $  F_{\bf{u}} \subseteq F \cap \bf{u}$ for any $\bf{u} \in \bf{M}_1 \cup \bf{O}_1$ and $|F \cap \bf{u}| = |F \cap \bf{v}|$ for any $\bf{uv} \in \bf{M}$. 
Moreover, if we extend our partial mapping $\phi$ of $W \cup \bigcup\fD''$ in such a way that the extended mapping satisfies $\phi(T) \cap F = \emptyset$ and $\phi$ is defined on the whole $T$ except of some $\bar \fD \subseteq \fD$ with $|\bigcup \bar\fD| \le \sqrt[4]{\eps}k$, then we can injectively extend $\phi$ to the whole tree $T$. 
\end{claim}

\begin{proof}
In the first part of the proof we propose a suitable procedure defining $F$. Then we show that during the defining procedure the size of $F$ is bounded by $\eta r \adeg(\bf{v}_1)/100$ as desired in the statement of the claim. Finally we use this fact to argue that the defining procedure finishes only after defining the whole $F$. In the second part of the proof we use Proposition \ref{prop:embed_regular_pair} to embed a small set of trees $\bar \fD$ in $F$. 

We define $F$ as follows. For each $\bf{u} \in \bf{M}_1 \cup \bf{O}_1$ we add $F_\bf{u}$ to $F$. 
Then for each set $F_\bf{u}$ we find a set of the same size in some neighbour $\bf{v}\not=\bf{v}_1$ of $\bf{u}$ and also add this set to $F$. 
We call this set $G_\bf{u}$ and find it as follows. For $\bf{uv} \in \bf{M}$ we take $G_\bf{u} = F_\bf{v}$. For $\bf{u} \in \bf{O}_1$ we find its neighbouring cluster in $\bf{O}_2 \cup \bf{M}_2$ with at least $\lfloor \eta r d|\bf{u}|/300 \rfloor$ vertices that were not yet added to $F$ and we set $G_\bf{u}$ to be this set (we argue later, why we always find a suitable neighbouring cluster, i.e., why this defining procedure cannot finish sooner than required). 
In the case when $F_\bf{u} \subseteq \bf{u} \in \bf{O}_1$, but $G_\bf{u} \subseteq \bf{v'} \in \bf{M}_2$, it is no longer true that $|F\cap\bf{u'}| = |F\cap\bf{v'}|$ for some matching edge $\bf{u'v'} \in \bf{M}$. We reestablish the condition by adding $\lfloor \eta r d |\bf{u}'|/300 \rfloor$ vertices from $\bf{u'}$ to $F$. 

During the defining procedure and after it finishes we have
\begin{align*}
    |F| 
    \le 3 \cdot \sum_{\bf{u} \in \bf{M}_1 \cup \bf{O}_1} \lfloor \eta r d|\bf{u}|/300 \rfloor 
    \le \eta r\adeg(\bf{v}_1)/100. 
\end{align*}

Now we argue that each cluster $\bf{u} \in \bf{O}_1$ has a neighbour in $\bf{M}_2 \cup \bf{O}_2$ with at least $\lfloor \eta r d|\bf{u}|/300 \rfloor $ vertices that are not yet in $F$. 
Since we know that 
\begin{align*}
    \adeg(\bf{u}, \bigcup\left(\bf{M}_2\cup \bf{O}_2\right)) 
    &\ge rk \\
    \JUSTIFY{$|F|/r \ll \adeg(\bf{v}_1) \le 2k$}&> 2|F| > 2\adeg(\bf{u}, F), 
\end{align*}
there is certainly a cluster $\bf{v}\in \bf{M}_2 \cup \bf{O}_2$ such that $\adeg(\bf{u}, \bf{v})/2 > \adeg(\bf{u}, F \cap \bf{v})$, hence $\adeg(\bf{u}, \bf{v} \setminus F) = \adeg(\bf{u}, \bf{v}) - \adeg(\bf{u}, \bf{v} \cap F) > \adeg(\bf{u}, \bf{v})/2 \ge d|\bf{v}|/2$, meaning that there is a subset of at least $d|\bf{v}|/2 > \lfloor \eta r d|\bf{v}|/300 \rfloor$ vertices in $\bf{v}$ that can be used to define $G_\bf{u}$.  

It remains to show how to embed any $\bar \fD$ of small size in $F$. We define the embedding $\phi$ of all trees $K \in \bar \fD$ in a step-by-step manner. Suppose that $\bf{u} \in \bf{M}_1 \cup \bf{O}_1$ and $G_\bf{u} \subseteq \bf{v}$.  
If the at most two neighbours $z_1, z_2$ of $K$ in $W$ are embedded to two vertices of $\bf{v}_1$ that are typical to set $F_\bf{u}$ and, moreover, $|\phi(T) \cap F_\bf{u}| \le \frac{d}2|F_\bf{u}|$ and $|\phi(T) \cap G_\bf{u}| \le \frac{d}2|G_\bf{u}|$, we can compute that for $i=1,2$ we have
\begin{align*}
    |F_\bf{u}\setminus \phi(T)|
    \ge \left(1-\frac{d}{2}\right)|F_\bf{u}|
    &\ge \frac12 \lfloor \eta r d|\bf{u}|/300 \rfloor\\
    \JUSTIFY{$\epsilon \ll (\eta r d)^2$}&\ge 4\sqrt{\eps}|\bf{u}|
\end{align*}
and similarly $|G_\bf{u}\setminus \phi(T)| \ge 4\sqrt{\eps}|\bf{v}|$. We also have
\begin{align*}
    |N_G(v_i)\cap(F_\bf{u}\setminus\phi(T))|
    &\ge |N_G(v_i)\cap F_\bf{u}| - |\phi(T)\cap F_\bf{u}|\\
    \JUSTIFY{$v_i$ is typical to $F_\bf{u}$}&\ge (d-\eps)|F_\bf{u}| - |\phi(T)\cap F_\bf{u}|\\
    \JUSTIFY{$\eps \ll d$,\; $|\phi(T) \cap F_\bf{u}| \le d|F_\bf{u}|/2$}&\ge \frac{d}{3} |F_\bf{u}|
    \ge \frac{d}{3} \lfloor \eta r d|\bf{u}|/300 \rfloor\\
    \JUSTIFY{$\epsilon \ll (\eta r d)^2$}&\ge 3\eps|\bf{u}|. 
\end{align*}
Hence in this case we can use Proposition~\ref{prop:embed_regular_pair} with parameters $U_{P\ref{prop:embed_regular_pair}} = (V(G) \setminus F) \cup \phi(T)$, $d_{P\ref{prop:embed_regular_pair}} = d$, $\eps_{P\ref{prop:embed_regular_pair}} = \eps$, $\beta_{P\ref{prop:embed_regular_pair}} = \beta$, $\bf{v}_{1,{P\ref{prop:embed_regular_pair}}} = \bf{v}_1$, $\bf{u}_{P\ref{prop:embed_regular_pair}} = \bf{u}$, $\bf{v}_{P\ref{prop:embed_regular_pair}} = \bf{v}$, $K_{P\ref{prop:embed_regular_pair}} = K$, $v_{1,{P\ref{prop:embed_regular_pair}}}= \phi(z_1)$, $v_{2,{P\ref{prop:embed_regular_pair}}} = \phi(z_2)$. The proposition then allows us to embed $K$. 

Now it suffices to show that for any $K$ we always find a suitable $\bf{u}$ such that $\phi(z_1), \phi(z_2)$ are typical to $F_\bf{u}$ and both $F_\bf{u}$ and $G_\bf{u}$ do not contain many embedded vertices of $T$. Recall that vertices $\phi(z_1), \phi(z_2)$ are typical to all but at most $\sqrt{\eps} m$ sets $F_\bf{u}$. 
If we cannot use for the embedding any other set $F_\bf{u}$ from remaining clusters of $\bf{M}_1 \cup \bf{O}_1$, it means that we have embedded more than $\frac{d}{2} \cdot\lfloor  \eta r d|\bf{v}_1|/300 \rfloor$ vertices to this set $F_\bf{u}$, or we have embedded at least the same number of vertices in the appropriate set $G_\bf{u}$. This means that the number of vertices we have embedded is at least 
\begin{align*}
    \left(|\bf{M}_1 \cup \bf{O}_1| - 2\sqrt{\eps}m \right) \cdot \left(\frac{d}{2} \cdot \lfloor \eta r d |\bf{v}_1|/300 \rfloor \right)
    &\ge \left(\frac{|\bf{M} \cup \bf{O}_1|}{2} - 2\sqrt{\eps}m \right) \cdot \frac{d^2 r \eta}{700}|\bf{v}_1|\\
    \JUSTIFY{$\adeg(\bf{v}_1) \le |\bf{M} \cup \bf{O}_1|\cdot |\bf{v}_1|$}&\ge \left( \frac{\adeg(\bf{v}_1)}{2} - 2\sqrt{\eps}m|\bf{v}_1|  \right) \cdot \frac{d^2 r \eta}{700} \\
    \JUSTIFY{$\eps \ll d^{10} r^5 \eta^5$}&\ge \left( \frac{\adeg(\bf{v}_1)}{2} - 2\sqrt{\eps}m|\bf{v}_1|  \right) \cdot \sqrt[5]{\eps} \\
    \JUSTIFY{$m|\bf{v}_1| \le n$}&\ge \left( \frac{k}{2} - 2\sqrt{\eps}n  \right) \cdot \sqrt[5]{\eps} \\
    \JUSTIFY{$k \ge \eta n/2$}&\ge \left(\frac12 - \frac{4\sqrt{\eps}}{\eta} \right) \sqrt[5]{\eps} k\\
    &> \sqrt[4]{\eps} k,
\end{align*}
a contradiction. 
\end{proof}

\subsubsection*{Embedding algorithm}

So far we have embedded the set $W$ in vertices of $\bf{v}_1$ that are typical to almost all clusters in the neighbourhood of $\bf{v}_1$.  We also embedded the small set $\fD''$ in $\bf{v}_2$. We denote this partial embedding by $\phi_0$ and we use the symbol $\phi$ to denote the successive extensions of $\phi_0$ that we will construct. 

We invoke Claim~\ref{claim:localES_finishing} to get a small set $F$. Now we will gradually embed microtrees from $\fD$ in $\bigcup \left(\bf{M} \cup \bf{O}\right) \setminus F$, until the number of vertices of microtrees that were not embedded yet is at most $\sqrt[4]{\eps} k$. Then we embed the remaining parts of $T$ in $F$ using Claim~\ref{claim:localES_finishing}. 

We will use the following notation for the sake of brevity.  
\begin{definition}
    Let $\phi$ be a fixed partial embedding of $T$ in $G$ extending $\phi_0$.
    We say that a cluster $\bf{u}$ is full, if 
    \begin{align*}
        |\bf{u} \cap \left( \phi(V(T)) \cup F  \right)| \ge \abs{\bf{u}} - 4\sqrt{\eps} \abs{\bf{u}}. 
    \end{align*}
    We say that a cluster $\bf{u} \in N_{\bf{G}}(\bf{v}_1)$ is saturated, if 
    \begin{align*}
        |\bf{u} \cap \left( \phi(V(T)) \cup F  \right)|\ge \adeg(\bf{v}_1, \bf{u}) - 4\sqrt{\eps} \abs{\bf{u}}. 
    \end{align*}
    We say that a matching edge $\bf{uv} \in \bf{M}$ is saturated, if 
    \begin{align*}
        |\left( \bf{u} \cup \bf{v} \right) \cap \left( \phi(V(T)) \cup F  \right)| \ge \adeg(\bf{v}_1, (\bf{u} \cup \bf{v})) - 8\sqrt{\eps} \abs{\bf{u}} - \beta k. 
    \end{align*}
\end{definition}

Note that every full cluster is also saturated. The intuition behind these definitions will be clear from the statements of the following claims. Recall that $V_1, V_2$ denote the two colour classes of $T$ and that $W \subseteq V_1$. 

\begin{claim}
\label{claim:localES_embed}
Let $\phi$ be a partial embedding of $T$ in $G$ extending $\phi_0$. 
If $\bf{u} \in N_{\bf{G}}(\bf{v}_1)$ is not saturated and $\bf{v} \in N_{\bf{G}}(\bf{u}) \setminus \{ \bf{v}_1 \}$ is not full, then, unless $\phi$ satisfies $|\dom(\phi)| \ge k - \sqrt[4]{\eps}k$, we may extend $\phi$ injectively to some $K \in \fD'$ that was not yet embedded in such a way that $\phi(V(K) \cap V_2) \subseteq \bf{u}$, $\phi(V(K) \cap V_1) \subseteq \bf{v}$, and $\phi(V(K)) \cap F = \emptyset$. 
\end{claim}

\begin{proof}
We have ensured that all trees of $\fD$ such that $\bf{u}$ is not nice to them have at most $\sqrt[4]{\eps} k$ vertices. Hence there is a yet nonembedded tree $K \in \fD$ such that its at most two neighbours $t_1, t_2$ in $W$ are embedded in vertices of $\bf{v}_1$ that are typical to $\bf{u}$. Note that this implies that 
\begin{align*}
    |N_G(t_i) \cap \left( \bf{u}\setminus\left(\phi(V(T)) \cup F\right)\right)|
    &\ge \adeg(\bf{v}_1, \bf{u}) - \eps|\bf{u}| - |\bf{u} \cap \left( \phi(V(T)) \cup F\right)|\\
    \JUSTIFY{$\bf{u}$ is not saturated}&\ge \adeg(\bf{v}_1, \bf{u}) - \eps|\bf{u}| - \left( \adeg(\bf{v}_1, \bf{u}) - 4\sqrt{\eps} \abs{\bf{u}}\right)\\
    &\ge 3\eps |\bf{u}|. 
\end{align*}

Recall that neither $\bf{u}$, nor $\bf{v}$ are full. Together with the inequality above this enables us to apply Proposition~\ref{prop:embed_regular_pair} with $d_{P\ref{prop:embed_regular_pair}} :=d, \eps_{P\ref{prop:embed_regular_pair}} =\eps, \beta_{P\ref{prop:embed_regular_pair}} =\beta, \bf{v}_{1,P\ref{prop:embed_regular_pair}}  = \bf{v}_1, \bf{u}_{P\ref{prop:embed_regular_pair}}  = \bf{u}, \bf{v}_{P\ref{prop:embed_regular_pair}}  = \bf{v}, K_{P\ref{prop:embed_regular_pair}} = K, v_{i, P\ref{prop:embed_regular_pair}} = \phi(t_i), x_{i, P\ref{prop:embed_regular_pair}} = N_T(t_i) \cap K, U_{P\ref{prop:embed_regular_pair}} = \phi(V(T)) \cup F$. 
The proposition then allows us to extend injectively $\phi$ to $K$. 
\end{proof}

\begin{claim}
Let $\phi$ be a partial embedding of $T$ in $G$ extending $\phi_0$.
    \label{claim:localES_saturation}
    \begin{enumerate}
        \item
        There exists either an unsaturated cluster of $\bf{O}_1$ or an unsaturated edge of $\bf{M}$. 
        \item
        Suppose that $\phi(V_2) \cap \bigcup \bf{M_2} = \emptyset$ and let $\bf{u} \in \bf{O}_1$. There exists a cluster in $N_{\bf{G}}(\bf{u}) \setminus \{ \bf{v}_1 \}$ that is not full. 
    \end{enumerate}
\end{claim}

\begin{proof}
    \begin{enumerate}
        \item 
        Suppose that each edge of $\bf{M}$ is saturated. For every $\bf{uv} \in \bf{M}$ we then have 
        \begin{align*}
            |\left( \bf{u} \cup \bf{v} \right) \cap \left( \phi(V(T)) \cup F  \right)|&\\
            \JUSTIFY{$\bf{uv}$ is saturated}&\ge \adeg(\bf{v}_1, \bf{u} \cup \bf{v}) - 8\sqrt{\eps} \abs{\bf{u}} - \beta k  \\
            \JUSTIFY{$\adeg(\bf{v}_1, \bf{u} \cup \bf{v}) \ge 2d|\bf{v}_1|$} &\ge \adeg(\bf{v}_1, \bf{u} \cup \bf{v}) \left( 1 - \frac{8\sqrt{\eps} \abs{\bf{v}_1} + \beta k}{2d |\bf{v}_1|} \right)\\
            \JUSTIFY{$|\bf{v}_1| \ge n/M_\reg(\eps)$, $k\le n$}&\ge \adeg(\bf{v}_1, \bf{u} \cup \bf{v}) \left( 1 - \frac{4\sqrt{\eps}}{d} - \frac{\beta n}{2d n/M_\reg(\eps)}\right)\\
            \JUSTIFY{$\eps \ll (\eta d)^2$,\; $\beta \ll d\eta /M_\reg(\eps)$}&\ge \adeg(\bf{v}_1, \bf{u} \cup \bf{v})(1-\eta/100)
        \end{align*}
        Suppose that each cluster of $\bf{O}_1$ is saturated. After a similar calculation, we get that for each $\bf{u} \in \bf{O}_1$ we have
        \begin{align*}
            |\bf{u} \cap (\phi(V(T)) \cup F)| 
            \ge \adeg(\bf{v}_1, \bf{u})(1-\eta/100). 
        \end{align*}
        Hence we have
        \begin{align*}
            \bigl\lvert\bigcup(\bf{M} \cup \bf{O}_1) \cap (\phi(V(T)) \cup F))\bigr\rvert
            &\ge \adeg(\bf{v}_1)(1-\eta/100)\\ 
            &= \eta \adeg(\bf{v}_1)/100 + \adeg(\bf{v}_1)(1-\eta/50)\\
            \JUSTIFY{$\adeg(\bf{v}_1) \ge k+\eta k / 20$}&\ge \eta \adeg(\bf{v}_1)/100 + (k+\eta k / 20)(1-\eta/50) \\
            \JUSTIFY{$|F| \le \eta \adeg(\bf{v}_1)/100$}&>|F|+k
            \ge |\phi(V(T)) \cup F|,
        \end{align*}
        a contradiction. 
        
        \item
        Suppose that all clusters in  $N_{\bf{G}}(\bf{u}) \setminus \{ \bf{v}_1 \}$ are full. 
        For each full cluster $\bf{v}$ we have 
        \begin{align*}
            |\bf{v} \cap (\phi(V(T)) \cup F)| 
            &\ge |\bf{v}|(1 - 4 \sqrt{\eps} ) \\
            \JUSTIFY{$\eps \ll \eta^2$}&\ge |\bf{v}|(1-\eta/200) \\
            &\ge \adeg(\bf{u}, \bf{v})(1-\eta/200).
        \end{align*}
        Since this holds for any $\bf{v} \in N_{\bf{G}}(\bf{u}) \setminus \{ \bf{v}_1 \}$, we get that 
        \begin{align*}
             \bigl\lvert\bigcup(\bf{M}_2 \cup \bf{O}_2) \cap (\phi(V(T)) \cup F))\bigr\rvert
             &\ge \adeg\left(\bf{u}, \bigcup(\bf{M}_2 \cup \bf{O}_2) \setminus \bf{v}_1\right)\left(1-\eta/200\right)\\ 
             &\ge \left( \adeg\left(\bf{u}, \bigcup(\bf{M}_2 \cup \bf{O}_2)\right) - |\bf{v}_1| \right)\left(1-\eta/200\right)\\ 
             \JUSTIFY{$N_{\bf G}(\bf u) \subseteq \{\bf{v}_1\} \cup \bf{M}_2 \cup \bf{O}_2$}&\ge \adeg\left(\bf{u}\right) \left(1-\eta/200\right) -  |\bf{v}_1| \\ 
             \JUSTIFY{$\frac{\eta}{200}\adeg(u) \ge \frac{\eta rk}{200} \ge \frac{\eta^2 rn}{200} \gg \eps n \ge  |\bf{v}_1|$}&\ge \adeg(\bf{u})(1-\eta/100). 
        \end{align*}
        Hence, we have
        \begin{align*}
            \bigl\lvert\bigcup(\bf{M}_2 \cup \bf{O}_2) \cap (\phi(V_1) \cup F))\bigr\rvert \\
           \JUSTIFY{$\phi(V_2) \cap \left( \bigcup (\bf{M}_2\cup\bf{O}_2) \right) = \emptyset$}&= \bigl\lvert\bigcup(\bf{M}_2 \cup \bf{O}_2) \cap (\phi(V(T)) \cup F))\bigr\rvert\\ 
            \JUSTIFY{inequality above} &\ge \adeg(\bf{u})(1-\eta/100) \\
            &\ge \eta\adeg(\bf{u})/50 + \adeg(\bf{u}) (1-\eta/30)\\
            \JUSTIFY{$\adeg(\bf{v}_1)\le2k$, $\adeg(\bf{u})\ge rk+\eta k/20$}&\ge \eta r \adeg(\bf{v}_1)/100 + (rk + \eta k /20)(1-\eta/30)\\
            &> |F| + rk \\
            &\ge |F| + |V_1| 
            \ge |(\phi(V_1) \cup F))|,
        \end{align*}
        a contradiction. 
    \end{enumerate}
 
\end{proof}

We can now finish the proof of Theorem \ref{thm:localES_dense_skew}. 

\begin{proof}
We will gradually embed microtrees from $\fD'$ in $\bigcup\left(\bf{M} \cup \bf{O}\right)$ in a specified manner using Claim~\ref{claim:localES_embed} (hence avoiding the set $F$), until $|\dom(\phi)| \ge k - \sqrt[4]{\eps}k$, or all edges of $\bf{M}$ and all vertices of $\bf{O}$ are saturated -- from Claim~\ref{claim:localES_saturation} (1) we know that the latter actually cannot be true. When $|\dom(\phi)| \ge k - \sqrt[4]{\eps}k$, we finish by applying Claim~\ref{claim:localES_finishing} on our set $F$. 
We split the embedding procedure into three phases:

\begin{enumerate}
    \item \textit{Phase 1 -- saturating the matching edges of $\bf{M}$. }
    In the first phase we embed gradually the microtrees of $\fD'$ in the edges of $\bf{M}$ in such a way that for each $K\in \fD'$ we have $\phi(K \cap V_2) \subseteq \bf{M}_1$. We run the process of applying Claim~\ref{claim:localES_embed} for each edge $\bf{uv}$ until either $\bf{u} \in \bf{M}_1$ is saturated, $\bf{v} \in \bf{M_2}$ is full, or $|\dom(\phi)| \ge k - \sqrt[4]{\eps}k$. 
    
    \item \textit{Phase 2 -- saturating the clusters in $\bf{O}$. }
    We repeatedly pick a cluster $\bf{v} \in \bf{O}_1$ and then embed trees from $\fD'$ in it by repeatedly applying Claim~\ref{claim:localES_embed} in such a way that for each embedded $K$ we have $\phi(K \cap V_2) \subseteq \bf{O}_1$ and $\phi(K \cap V_1) \subseteq \bf{M}_2 \cup \bf{O}_2$. Note that due to Claim~\ref{claim:localES_saturation} (2) the cluster $\bf{v}$ has always a neighbour that is not full and can be, thus, used for the embedding. Hence we can apply this procedure until all clusters from $\bf{O}_1$ are saturated, or $|\dom(\phi)| \ge k - \sqrt[4]{\eps}k$. 
    
    \item \textit{Phase 3 -- finalising the matching $\bf{M}$. }
    All clusters in $\bf{O}_1$ are now saturated.
    Our task is now to show how to saturate the remaining edges of $\bf{M}$. This may not be possible with current $\phi$ as it is defined right now, since it could have for example happened that after the first phase we completely filled one cluster from a matching pair, while the other cluster remained almost empty. 
    We solve this problem by potentially redefining the embedding of several microtrees that were embedded in $\bf{M}_1 \cup \bf{M_2}$ in Phase 1. 
    
    Note that for each edge $\bf{uv} \in \bf{M}$, $\bf{u} \in \bf{M}_1$, it is true that either $\bf{u}$ is saturated, or $\bf{v}$ is full at the end of Phase 1. 
    We deal with the first case in part (a). 
    In the latter case we did not embed anything in $\bf{v}$ in Phase 2. We undefine embedding of all trees that were embedded in $\bf{uv}$ and saturate this edge in part (c). 
    \begin{enumerate}
        \item 
        If $\bf{u}$ is saturated, we repeatedly embed trees in $\bf{uv}$ in such a way that for each $K \in \fD'$ we have $\phi(K \cap V_2) \subseteq \bf{v}$. We do this until either $\bf{u}$ is full, or $\bf{v}$ is saturated. In the latter case the whole edge is saturated. We deal with the first case in (b). 
        
        \item
        Suppose that $\bf{u}$ is full, but $\bf{v}$ is not saturated. Note that Claim~\ref{claim:localES_finishing} ensures that $|F \cap \bf{u}| = |F \cap \bf{v}|$. Hence it must be the case that $\abs{\phi(V(T)) \cap \bf{u}} \ge \abs{\phi(V(T)) \cap \bf{v}}$. Moreover, in Phase 2 we did not embed trees in $\bf{u}$. This means that there exists a tree $K \in \fD'$ that was embedded in the matching edge $\bf{uv}$ in such a way that $\abs{\phi(V(K)) \cap \bf{u}} \ge \abs{\phi(V(K)) \cap \bf{v}}$. 
        As long as it is true that $|\phi(V(T)) \cap \bf{u}| \ge |\phi(V(T)) \cap \bf{v}|$, we find any tree $K$ with this property and we undefine its embedding. When this procedure ends, we have $\left| \left|\phi(V(T)) \cap \bf{u}\right| - \left|\phi(V(T)) \cap \bf{v}\right| \right| \le \beta k$. We later refer to this inequality as the \emph{balancing} condition. 
        
        \item
        Finally, it suffices to show how to saturate an edge $\bf{uv}$ fulfilling the balancing condition (note that if $\phi(V(T)) \cap \bf{uv} = \emptyset$, then the matching edge certainly fulfills the condition). We again embed the microtrees in $\bf{uv}$ one after another. Unless one of the clusters is saturated, we choose to embed $K \in \fD'$ in such a way that the colour class of $K$ with less vertices is embedded in the cluster such that more of its vertices were already used for the embedding of $T$. In this way we ensure that the balancing condition still holds. 
    
        After one cluster, say $\bf{u}$, becomes saturated, we continue by embedding only in such a way that for each $K \in \fD'$ we have $\phi(K \cap V_2) \subseteq \bf{v}$. We do this until either $\bf{v}$ becomes saturated, or $\bf{u}$ is full. In the first case the whole edge $\bf{uv}$ is clearly saturated. In the other case note that we have 
        \begin{align*}
            |(\phi(V(T))\cup F) \cap \bf{u}| 
            &\\
            \JUSTIFY{$\bf{u}$ is full}&\ge |\bf{u}| - 4\sqrt{\eps} |\bf{u}|\\ \JUSTIFY{$|\bf{u}| = |\bf{v}| \ge \adeg(\bf{v}_1, \bf{v})$}&\ge \adeg(\bf{v}_1, \bf{v}) - 4\sqrt{\eps} |\bf{v}|
        \end{align*}
        and hence
        \begin{align*}
            |(\phi(V(T))\cup F) \cap \bf{v}| 
            &\\
            \JUSTIFY{the balancing condition}&\ge |(\phi(V(T))\cup F) \cap \bf{u}| - \beta k\\ 
            \JUSTIFY{$\bf{u}$ is full, hence saturated}&\ge \adeg(\bf{v}_1, \bf{u}) - 4\sqrt{\eps} |\bf{u}| - \beta k.
        \end{align*}
        Hence, the whole matching edge is saturated.     
    \end{enumerate}
\end{enumerate}

We described an algorithm that terminates when $|\dom(\phi)| \ge k - \sqrt[4]{\eps}k$, or all edges of $\bf{M}$ and all vertices of $\bf{O}$ are saturated. But the latter cannot happen due to Claim~\ref{claim:localES_saturation} (1). We finish by invoking Claim~\ref{claim:localES_finishing}. 
\end{proof}

\section{Acknowledgements}
\label{sec:conclusion}
This paper, as well as some similar results are part of the author's Bachelor's thesis \cite{Rozhon2018}. 

I would like to thank Stephan Wagner for providing the proof of Proposition \ref{prop:random_trees}. My great thanks go to Tereza Klimošová and Diana Piguet for many helpful discussions and comments and to two anonymous referees for many helpful comments. 

\bibliographystyle{siamplain}
\bibliography{bibliography}

\end{document}